\documentclass[12pt]{article}
\usepackage{amsmath}
\usepackage{amssymb}
\usepackage{amscd}
\usepackage{amsthm}

\theoremstyle{plain}
\newtheorem{Thm}{Theorem}

\newtheorem{Prop}[Thm]{Proposition}
\newtheorem{Cor}[Thm]{Corollary}
\newtheorem{Lem}[Thm]{Lemma}

\theoremstyle{definition}

\theoremstyle{Remark}
\newtheorem{Rem}[Thm]{Remark}
\newtheorem{Que}[Thm]{Question}

\numberwithin{equation}{section}

\title{Derived categories of toric varieties III}
\author{Yujiro Kawamata}

\begin{document}

\maketitle

\section{Introduction}

This is a continuation of papers \cite{toric} and \cite{toricII}.
We have studied the effect of each elementary birational map in the minimal model program to the derived categories 
in the case of toric pairs.
In this paper we consider problems of more global nature.
We add some cases to the list of affirmative answers to the 
derived McKay correspondence conjecture and its generalization, 
the \lq\lq $K$ implies $D$ conjecture''.

First we remark that the results in \cite{toric} and \cite{toricII} are 
easily extended to the relative situation (Theorem~\ref{relative}).
Then we remark that the derived McKay correspondence holds for any abelian quotient singularity 
in the following sense; 
the derived category defined by a finite abelian group contained in a general linear group 
is semi-orthogonally decomposed 
into derived categories of a relative minimal model and subvarieties of the quotient space 
(Theorem~\ref{abelian}).  
As an application, we prove that such derived McKay correspondence also holds for arbitrary quotient singularities in dimension $2$ (Theorem~\ref{GL2}).

Next we prove that any proper birational morphism between toric pairs which is a $K$-equivalence
is decomposed into a sequence of flops (Theorem~\ref{decomposition}).
This is a generalization, in the toric case, of a theorem in \cite{flop}, which states that 
any birational morphism between minimal models is decomposed into a sequence of flops.
We note that the latter case is easier because the log canonical divisor stays 
at the bottom and cannot be decreased when the models are minimal, 
while it is more difficult to preserve the level of the log canonical divisor in the former case. 

Now we state the theorems in details.
We refer the terminology to the next section.

The following is the reformulation of the results of \cite{toric} and \cite{toricII} to the relative case: 

\begin{Thm}\label{relative}
Let $(X,B)$ and $(Y,C)$ be $\mathbf{Q}$-factorial toric pairs whose coefficients 
belong to the standard set 
$\{1 - 1/m \mid m \in \mathbf{N}\}$, 
let $\tilde X$ and $\tilde Y$ be smooth Deligne-Mumford stacks associated to 
the pairs $(X,B)$ and $(Y,C)$, respectively, and
let $f: X \dashrightarrow Y$ be a toric rational map.
Assume one of the following conditions:

(a) $f$ is the identity morphism, and $B \ge C$.

(b) $f$ is a flip, $C=f_*B$, and $K_X+B \ge K_Y+C$.

(c) $f$ is a divisorial contraction, $C=f_*B$, and $K_X+B \ge K_Y+C$.

(d) $f$ is a divisorial contraction, $C=f_*B$, and $K_X+B \le K_Y+C$.

(e) $f$ is a Mori fiber space and $C$ is determined by $(X,B)$ and $f$ as explained in \cite{toric}. 

Then the following hold:

(a), (b), (c): There are toric closed subvarieties $Z_i$ ($1 \le i \le l$) of $Y$ such that $Z_i \ne Y$ 
for some $l \ge 0$,  
and fully faithful functors $\Phi: D^b(\text{coh}(\tilde Y)) \to D^b(\text{coh}(\tilde X))$
and $\Psi_i: D^b(\text{coh}(\tilde Z_i)) \to D^b(\text{coh}(\tilde X))$
for the smooth Deligne-Mumford stacks $\tilde Z_i$ associated to the $Z_i$ 
such that there is a semi-orthogonal decomposition of triangulated categories 
\[
D^b(\text{coh}(\tilde X)) = \langle \Psi_1(D^b(\text{coh}(\tilde Z_1))), \dots,
\Psi_l(D^b(\text{coh}(\tilde Z_l))), \Phi(D^b(\text{coh}(\tilde Y))) \rangle
\]

(d): There are $Z_i \subset Y$ ($1 \le i \le l$) as in (c), 
and fully faithful functors $\Phi: D^b(\text{coh}(\tilde X)) \to D^b(\text{coh}(\tilde Y))$
and $\Psi_i: D^b(\text{coh}(\tilde Z_i)) \to D^b(\text{coh}(\tilde Y))$
such that there is a semi-orthogonal decomposition
\[
D^b(\text{coh}(\tilde Y)) = \langle \Psi_1(D^b(\text{coh}(\tilde Z_1))), \dots,
\Psi_l(D^b(\text{coh}(\tilde Z_l))), \Phi(D^b(\text{coh}(\tilde X))) \rangle
\]

(e):  There are fully faithful functors $\Phi_j: D^b(\text{coh}(\tilde Y)) \to D^b(\text{coh}(\tilde X))$ 
($1 \le j \le m$) for some $m \ge 2$ and a semi-orthogonal decomposition
\[
D^b(\text{coh}(\tilde X)) = \langle \Phi_1(D^b(\text{coh}(\tilde Y))), \dots, 
\Phi_m(D^b(\text{coh}(\tilde Y))) \rangle
\]

Moreover if $K_X+B = K_Y+C$ in the cases (b) or (c), then $\Phi$ is an equivalence.
\end{Thm} 

We note that, when $f$ is a divisorial contraction, 
the direction of the inclusion of the derived categories 
is unrelated to the direction of the morphism, 
but is the same as the direction of the inequality of the canonical divisors.

If $X$ and $Y$ are smooth and $f$ is a blowing up of a smooth center in the case (b), or if  
$X$ and $Y$ are smooth and $f$ is a standard flip in the case (c), then
it is a theorem by Bondal-Orlov (\cite{BO}).
The derived equivalence in the case of flops (case (c) with $K_X+B = K_Y+C$) was 
proved for general (non-toric) case under the additional assumptions as follows:
if $\dim X = 3$ and $X$ is smooth by Bridgeland (\cite{Bridgeland}),
if $\dim X = 3$ and $X$ has only Gorenstein terminal singularities by Chen (\cite{Chen}) 
and Van den Bergh (\cite{VdBergh}), 
if $\dim X = 3$ and $X$ has only terminal singularities by \cite{DK}, 
and if $X$ is a holomorphic simplectic manifold by Kaledin (\cite{Kaledin}).
There are also interesting results concerning the last case by Cautis (\cite{Cautis}) by using the 
categorification of linear group actions, 
and by Donovan-Segal (\cite{DS}), Ballard-Favero-Katzarkov (\cite{BFK}) and
Halpern-Leistner (\cite{HL}) by using the theory of variations of GIT quotients.

As a corollary we obtain

\begin{Cor}\label{combine}
Let $(X,B)$ be a $\mathbf{Q}$-factorial toric pair whose coefficients belong to the standard set 
$\{1 - 1/m \mid m \in \mathbf{N}\}$.
Assume that there is a projective toric morphism $f: X \to Z$ to another  
$\mathbf{Q}$-factorial toric variety.
Then there exist toric closed subvarieties $Z_i$ ($1 \le i \le m$) of $Z$ 
for some $m \ge 1$ such that there are fully faithful functors 
$\Phi_i: D^b(\text{coh}(\tilde Z_i)) \to D^b(\text{coh}(\tilde X))$ with
a semi-orthogonal decomposition 
\[
D^b(\text{coh}(\tilde X)) \cong \langle \Phi_1(D^b(\text{coh}(\tilde Z_1))), \dots, 
\Phi_m(D^b(\text{coh}(\tilde Z_m))) \rangle
\]
where $\tilde X$ and the $\tilde Z_i$ are smooth Deligne-Mumford stacks associated 
to $(X,B)$ and the $Z_i$. 
\end{Cor}

In other words, one can say that the derived category 
$D^b(\text{coh}\tilde X)$ is generated by a {\em relative exceptional collection}.
The proofs of the above statements are the same as in \cite{log crepant}, 
\cite{toric} and \cite{toricII} except Theorem~\ref{relative}~(a).

The following is the derived McKay correspondence for finite abelian groups:

\begin{Thm}\label{abelian}
Let $G \subset GL(n,\mathbf{C})$ be a finite abelian subgroup acting naturally on an affine 
space $\mathbf{A} = \mathbf{C}^n$, and let $X=\mathbf{A}/G$ be the quotient space.
Let $f: Y \to X$ be a $\mathbf{Q}$-factorial terminal relative minimal model of $X$.
Then there exist toric closed subvarieties $Z_i$ ($1 \le i \le m$) of $X$ with $Z_i \ne X$ 
for some $m \ge 0$, with possible repetitions, 
such that there is a semi-orthogonal decomposition 
\[
D^b(\text{coh}([\mathbf{A}/G])) \cong \langle D^b(\text{coh}(\tilde Z_1)), \dots, 
D^b(\text{coh}(\tilde Z_m)), D^b(\text{coh}(\tilde Y)) \rangle
\]
where $\tilde Y$ and the $\tilde Z_i$ are smooth Deligne-Mumford 
stacks associated to $Y$ and the $Z_i$, respectively. 
Moreover if $G \subset SL(n,\mathbf{C})$, then $m=0$.
\end{Thm}

In the case of dimension $2$, we do not need the assumption that the group is abelian:

\begin{Thm}\label{GL2}
Let $G \subset GL(2,\mathbf{C})$ be a finite subgroup acting naturally on an affine 
space $\mathbf{A} = \mathbf{C}^2$, let $X = \mathbf{A}/G$, and
let $f: Y \to X$ be the minimal resolution.
Then there exist smooth closed subvarieties $Z_i$ ($1 \le i \le m$) of $X$ with $Z_i \ne X$ 
for some $m \ge 0$, with possible repetitions, 
such that there is a semi-orthogonal decomposition 
\[
D^b(\text{coh}([\mathbf{A}/G])) \cong \langle D^b(\text{coh}(Z_1)), \dots, D^b(\text{coh}(Z_m)),
D^b(\text{coh}(Y)) \rangle.
\]
Moreover if there are no quasi-reflections in $G$, the elements whose invariant 
subspaces are of codimension $1$, then $\dim Z_i = 0$ for all $i$.
Thus the semi-orthogonal complement of $D^b(\text{coh}(Y))$
in $D^b(\text{coh}([\mathbf{A}/G]))$ is generated by an exceptional collection in this case.
\end{Thm}

The derived McKay correspondence was already proved for finite subgroups of 
$SL(2,\mathbf{C})$, $SL(3,\mathbf{C})$ (Bridgeland-King-Reid \cite{BKR}), 
and $Sp(2n,\mathbf{C})$ (Bezrukavnikov-Kaledin \cite{BK}).

Finally, we prove the following decomposition theorem for $K$-equivalent toric birational map:

\begin{Thm}\label{decomposition}
Let $f: X \dashrightarrow Y$ be a toric birational map between
projective $\mathbf{Q}$-factorial toric varieties which is an isomorphism in codimension $1$, and
let $B$ be a toric $\mathbf{R}$-divisor on $X$
whose coefficients belong to the interval $(0,1)$, and let $f_*B=C$.
Assume that $K_X+B = K_Y+C$.
Then the birational map $f: (X,B) \dashrightarrow (Y,C)$ is decomposed into a 
sequence of flops.
\end{Thm}

As a corollary of Theorems~\ref{relative} and \ref{decomposition}, 
we obtain an affirmative answer to 
\lq\lq $K$ implies $D$ conjecture'' in the toric case:

\begin{Cor}
Assume additionally that the coefficients of $B$ belong to a set
$\{1 - 1/m \mid m \in \mathbf{Z}_{>0}\}$.
Let $\tilde X$ and $\tilde Y$ be the smooth Deligne-Mumford staks 
associated to the pairs $(X,B)$ and $(Y,C)$, respectively.
Then there is an equivalence of triangulated categories
$D^b(\text{coh}(\tilde X)) \cong D^b(\text{coh}(\tilde Y))$.
\end{Cor}

\section{Preliminaries}

We fix the terminology in this section.
A pair $(X,B)$ consisting of a normal algebraic variety 
and an $\mathbf{R}$-divisor
is said to be {\em KLT} (resp. {\em terminal}) 
if there is a projective birational morphism 
$p: Z \to X$ from a smooth varity with an $\mathbf{R}$-divisor $D$ 
whose support is a normal crossing divisor such that 
an equality $p^*(K_X+B)=K_Z+D$ holds 
and all the coefficients of $D$ are smaller than $1$
(resp. all the coefficients of $D - p^{-1}_*B$ are smaller than $0$),
where the canonical divisors $K_X$ and $K_Z$ are defined by using the 
same rational differential forms.
It is called {\em $\mathbf{Q}$-factorial} if any prime divisor on $X$ is a 
$\mathbf{Q}$-Cartier divisor.
 
A birational map $h: X \dashrightarrow Y$ between varieties is said to be {\em proper}
if there exists a third variety $Z$ with proper birational morphisms $f: Z \to X$ and 
$g: Z \to Y$ such that $g = h \circ f$. 

Let $f: (X,B) \dashrightarrow (Y,C)$ be a proper birational map between KLT pairs.
We say that there is an {\em equality of log canonical divisors}
$K_X+B = K_Y+C$
(resp. {\em an inequality} $K_X+B \ge K_Y+C$), 
if $p^*(K_X+B) = q^*(K_Y+C)$
(resp. $p^*(K_X+B) \ge q^*(K_Y+C)$) 
as an ineqality (resp. an equality) of $\mathbf{R}$-divisors for some $p,q$ as above.
We note that such equality or inequality can be defined only 
when the birational map 
$f$ is fixed, and the definition is independent of the choice of $p,q$.
Such an equality (resp. an inequality) is called a {\em $K$-equivalence} 
(resp. {\em $K$-inequality}).
$f$ is also said to {\em crepant} if it is a $K$-equivalence.

A proper birational map $f: X \dashrightarrow Y$ is said to be 
{\em isomorphic in codimension $1$}
(resp. {\em surjective in codimension $1$}) if it induces a bijection
(resp. surjection) between the sets of prime divisors.

A {\em flop} (resp. {\em flip}) is a proper birational map
$f: (X,B) \dashrightarrow (Y,C)$ between 
$\mathbf{Q}$-factorial KLT pairs satisfying the following conditions: 
$f_*B=C$, $K_X+B=K_Y+C$ (resp. $K_X+B > K_Y+C$), and 
there exist projective birational morphisms 
$s: X \to W$ and $t: Y \to W$ to a normal variety 
which are isomorphisms in codimension $1$ and the relative Picard numbers 
$\rho(X/W)$ and $\rho(Y/W)$ are equal to $1$.

A {\em divisorial contraction} is a projective 
birational morphism $f: (X,B) \to (Y,C)$ between 
$\mathbf{Q}$-factorial KLT pairs satisfying the following conditions: 
$f_*B=C$, and the exceptional locus of $f$ consists of a single prime divisor.
We have always $\rho(X/Y)=1$, and 
we have either $K_X+B=K_Y+C$, $K_X+B > K_Y+C$, or $K_X+B < K_Y+C$.
The inverse birational map $f^{-1}: (Y,C) \dashrightarrow (X,B)$ of a 
divisorial contraction is called a {\em divisorial extraction}.

A {\em Mori fiber space } is a projective morphism $f: (X,B) \to Y$ from a  
$\mathbf{Q}$-factorial KLT pair to a normal variety satisfying the following conditions: 
$\rho(X/Y)=1$, $-(K_X+B)$ is ample, and $\dim X > \dim Y$.

For a $\mathbf{Q}$-factorial KLT pair $(X,B)$, there exists a sequence of 
crepat divisorial contractions 
\[
(Y,C)=(X_m,B_m) \to \dots \to (X_1,B_1) \to (X_0,B_0)=(X,B)
\]
such that $(Y,C)$ is terminal and $\mathbf{Q}$-factorial (\cite{BCHM}).
The composition $f: (Y,C) \to (X,B)$ is called a {\em $\mathbf{Q}$-factorial crepant terminalization} 
of the pair $(X,B)$.
We note that even if $B$ has standard coefficients, $C$ may not be so.

Let $f: (Y,C) \to (X,B)$ be a projective birational morphism
from a $\mathbf{Q}$-factorial terminal pair to a KLT pair. 
It is called a {\em $\mathbf{Q}$-factorial terminal relative minimal model} 
if $C=f_*^{-1}B$ and $K_Y+C$ is relatively nef over $X$.
We note that $f$ is not necessarily crepant.
The existence of such a model is also proved in \cite{BCHM}.

A rational number is said to be {\em standard} if it is of the form  
$1-1/m$ for a positive integer $m$.
A KLT pair $(X,B)$ is said to be of {\em quotient type}, if 
there is a quasi-finite surjective morphism $\pi': X' \to X$ from a 
smooth variety $X'$, which is not necessarily irreducible, such that 
$K_{X'}=(\pi')^*(K_X+B)$.
In this case $B$ has only standard coefficients, and 
$X$ is automatically $\mathbf{Q}$-factorial.
If the pair $(X,B)$ is toric, $X$ is $\mathbf{Q}$-factorial, and if 
$B$ has standard coefficients, then it is of quotient type.

There is a {\em smooth Deligne-Mumford stack} $\tilde X$ 
associated to the pair $(X,B)$ of quotient type in such a way that
a sheaf $\mathcal{F}$ on $\tilde X$ is nothing but a sheaf $\mathcal{F}'$ on $X'$  
with an isomorphism $p_1^*\mathcal{F}' \cong p_2^*\mathcal{F}'$ satisfying the cocycle condition, 
where $p_i: X' \times_X X' \to X'$ ($i=1,2$) are projections.
In this case, there is a finite birational morphism 
$\pi: \tilde X \to X$ such that $\pi^*(K_X+B)=K_{\tilde X}$.

For an algebraic variety or a Deligne-Mumford stack $X$, 
let $D^b(\text{coh}(X))$ 
denotes the bounded derived category of coherent sheaves on $X$.
There are variants of this notation; if a finite group $G$ acts on $X$, then
$D^b(\text{coh}^G(X))$ 
denotes the bounded derived category of $G$-equivariant coherent 
sheaves on $X$.
Thus $D^b(\text{coh}^G(X)) \cong D^b(\text{coh}([X/G]))$.

A triangulated category $\mathcal{T}$ is said to have a {\em semi-orthogonal decomposition}
$\mathcal{T} = \langle \mathcal{T}_1, \dots, \mathcal{T}_m \rangle$ by 
triangulated full subcategories $\mathcal{T}_i$ if the 
following conditions are satisfied:
(a) $\text{Hom}(a,b) = 0$ for $a \in \mathcal{T}_i$ and $b \in \mathcal{T}_j$ if $i > j$, and 
(b) for any object $a \in \mathcal{T}$ there is a sequence of objects $a_i \in \mathcal{T}$ 
and $b_i \in \mathcal{T}_i$ for $1 \le i \le m$ such that 
$a=a_1$ and there are distinguished triangles $a_{i+1} \to a_i \to b_i$, where
$a_{m+1}=0$.
In particular, if the $\mathcal{T}_i$ are equivalent to $D^b(\text{Spec }k)$ for all $i$, then 
$\mathcal{T}$ is said to have a {\em full exceptional collection}.

A {\em toric pair} $(X,B)$ consists of a toric variety and an $\mathbf{R}$-divisor 
whose irreducible components are invariant under the torus action.
$X$ is $\mathbf{Q}$-factorial if and only if the corresponding fan is simplicial.
In this case it has only abelian quotient singularities.
Conversely, the quotient of an affine space by any finite abelian group is a 
$\mathbf{Q}$-factorial toric variety. 
Any $\mathbf{Q}$-factorial crepant terminalization and 
$\mathbf{Q}$-factorial terminal relative minimal model of a toric pair is toric.

For a pair of quotient type, we can define a derived category whose homological dimension is finite.
Therefore we would like to ask the following:

\begin{Que}\label{quotient-flop}
Let $f: (X,B) \dashrightarrow (Y,C)$ be a flop or a flip.
If $(X,B)$ is of quotient type, then is $(Y,C)$ too?
\end{Que}

We note that a divisorial contraction of a smooth $3$-fold may produce a 
singularity of non-quotient type even if it is crepant.
We note also that even if $X$ is smooth and $f$ is a flop, $Y$ is not necessarily smooth, but we have
still the derived equivalence in some examples when we consider the associated 
Deligne-Mumford stacks (cf. \cite{Francia}).

The following lemma is standard:

\begin{Lem}\label{codimension}
Let $f: (X,B) \dashrightarrow (Y,C)$ be a proper birational map between 
$\mathbf{Q}$-factorial terminal pairs, and assume that
$K_X+B \ge K_Y+C$. 
Then $f$ is surjective in codimension $1$, and $f_*B \ge C$.
Moreover, if $K_X+B = K_Y+C$, then $f$ is an isomorphism in 
codimension $1$.
\end{Lem}

\begin{proof}
Suppose that there is a prime divisor $C_1$ on $Y$ 
which is not the strict transform of a prime
divisor on $X$.
Let $c_1 \ge 0$ be the coefficient of $C_1$ in $C$.
Let $p: Z \to X$ and $q: Z \to Y$ be projective birational morphisms 
from a smooth variety, and let $D_1$ be the strict transform of $C_1$.
Since we have an equality $p^*(K_X+B) \ge q^*(K_Y+C)$, 
the coefficient of $D_1$ in $p^*(K_X+B)-K_Z$ 
should be non-negative, a contradiction to the fact that $(X,B)$ is terminal.
Therefore $f$ is surjective in codimension $1$.
If we apply the first assertion to $f$ and $f^{-1}$, then we obtain the second assertion.
\end{proof}

\section{Decomposition theorem for toric pairs}

We prove Theorem~\ref{decomposition} in this section.

First we prove a condition for the decomposition theorem that is also valid for the non-toric case.
Let $f: (X,B) \dashrightarrow (Y,C)$ be a birational map between projective $\mathbf{Q}$-factorial 
KLT pairs which is an isomorphic in codimension 1 and such that $K_X+B=K_Y+C$. 
Let $p:Z \to X$ and $q: Z \to Y$ be projective birational morphisms
such that $f \circ p = q$.
Let $\{l_{\lambda}\}$ be a set of curves on $X$ whose classes generate the 
cone of curves $\overline{NE}(X)$ as a closed convex cone.
Let $L_Y$ be an ample and effective divisor on $Y$ 
such that $L_Y-(K_Y+C)$ is also ample, and 
let $L_X = f_*^{-1}L_Y$ be its strict transform.

\begin{Lem}\label{condition}
Assume that, if $(L_X,l_{\lambda}) < 0$ for some $\lambda$, 
then $((K_X+B),l_{\lambda}) = 0$.
Then $f$ is decomposed into a sequence of flops.  
\end{Lem}

\begin{proof}
We shall find suitable extremal rays
in order to find a path from $(X,B)$ to $(Y,C)$.

Assume first that $L_X$ is nef.
We claim that $L_X-(K_X+B)$ is also nef.
Otherwise, there exists a curve $l_{\lambda}$ such that 
$((L_X-(K_X+B)),l_{\lambda})<0$.
By the assumption, it follows that $((K_X+B),l_{\lambda})=0$.
Therefore $(L_X,l_{\lambda})<0$, a contradiction.

By the base point free theorem (\cite{KMM}), $L_X$ is semiample.
Since $f$ is an isomorphism in codimension $1$, 
the associated morphism coincides with $f$.
Since both $X$ and $Y$ are $\mathbf{Q}$-factorial, we conclude that $f$ is an
isomorphism.

Next assume that $L_X$ is not nef.
We take a small positive number $\epsilon$ such that 
the pair $(X,B+\epsilon L_X)$ is KLT.
We would like to run a minimal model program for this pair in order to
move from the pair $(X,B)$ to the pair $(Y,C)$.
Here we have to note that there may be extremal rays of $(X,B+\epsilon L_X)$ 
which do not come from the negativity of $L_X$.

Since $L_X$ is not nef, 
there exists a curve $l_{\lambda}$ such that $(L_X,l_{\lambda}) < 0$.
By the assumption, $((K_X+B),l_{\lambda})=0$ for such curve.
Therefore $K_X+B+\epsilon L_X$ is not nef.

The part of the closed cone of curves $\overline{NE}(X)$ 
where $L_X$ is negative is contained in the part where $K_X+B+\epsilon L_X$ 
is negative.
Thus there exists an extremal ray $R$ of $\overline{NE}(X)$ on which 
$L_X$ is negative.
Since $K_X+B$ is numerically trivial on $R$.
the flip of $R$ for the pair $(X,B+\epsilon L_X)$ is a flop 
for $(X,B)$.
Since $B+\epsilon L_X$ is big, this process terminates by \cite{BCHM}, and 
$f$ is decomposed into the flops.
\end{proof}

It is important note that $\{l_{\lambda}\}_{\lambda \in \Lambda}$ is not necessarily the 
set of all curves.

Now we consider the toric case.
We take $p: Z \to X$, $q: Z \to Y$ and $L_Y$ to be also toric.
Let $\Delta_X$ and $\Delta_Y$ be the simplicial fans corresponding to 
$X$ and $Y$, respectively, in the same vector space $N_{\mathbf{R}}$.
By assumption, the sets of primitive vectors on the edges of 
$\Delta_X$ and $\Delta_Y$ coincide.
Since $\overline{NE}(X)$ is generated by toric curves, 
Theorem~\ref{decomposition} is a consequence of the following:

\begin{Lem}
(1) Let $w$ be a wall in $\Delta_X$ and let $l_w$ be the corresponding curve 
on $X$.
Assume that $(L_X,l_w) \le 0$.
Then $w$ is not a wall in $\Delta_Y$.

(2) Let $w$ be a wall in $\Delta_X$ which does not belong to $\Delta_Y$.
Then $((K_X+B),l_w)=0$.
\end{Lem}

\begin{proof}
(1) Since $f$ is an isomorphism in codimension $1$, 
the set of primitive vectors $\{v_i\}$ on the edges of $\Delta_X$ and $\Delta_Y$ coincide.
Let $g_X$ and $g_Y$ be functions on $N_{\mathbf{R}}$ corresponding to the divisors
$L_X$ and $L_Y$, respectively.
$g_X$ and $g_Y$ have the same values at the $v_i$, 
and are linear inside each simplex 
belonging to $\Delta_X$ and $\Delta_Y$, respectively.
Since $L_Y$ is ample, $g_Y$ is convex along walls of $\Delta_Y$.

Assuming that $w$ is also a wall in $\Delta_Y$, 
we shall derive a contradiction.
Under this assumption, the values of $g_X$ and $g_Y$ are equal on $w$.

Let $v_1,v_2$ be two vertexes adjacent to $w$ in $\Delta_X$.
We take a general point $P$ on $w$ such that the line segments $I_i$ 
connecting $P$ to $v_i$ for $i=1,2$ pass only the interiors of simplexes and 
walls of $\Delta_Y$.
Then the function $g_Y$ is convex along the $I_i$.

Since $g_X$ is linear on the $I_i$ and coincides with $g_Y$ at the end points,
we have an inequality $g_X \ge g_Y$ on $I_1 \cup I_2$.
On the other hand, $g_X = g_Y$ on $w$.
Since $g_Y$ is convex along $w$, we conclude that $g_X$ is also convex along 
$w$, hence $(L_X,l_w) > 0$, a contradiction.

(2) Let $D_i$ be the the prime divisors on $X$ corresponding the $v_i$.
We write $B = \sum d_iD_i$, and let $h$ be the function on the vector space 
$N_{\mathbf{R}}$ which takes values $1-d_i$ at the $v_i$ and linear 
inside each simplex of $\Delta_X$.
Since $K_X+B=K_Y+C$, $h$ is also linear inside each simplex of $\Delta_Y$.
There is a point inside $w$ which is an interior point of a simplex in 
$\Delta_Y$.
Then $h$ is linear across $w$.
Therefore $((K_X+B),l_w)=0$.
\end{proof}

\begin{Rem}
The relative version is similarly proved.
We need to assume that $X$ and $Y$ are projective and toric over a toric variety $S$ and 
$f$ is a toric birational map over $S$.
We only consider curves relative over $S$, i.e., those mapped to single points on $S$.
\end{Rem}

If $\dim X = 3$, then the condition of Lemma~\ref{condition} is easily verified
even in the non-toric case:

\begin{Prop}\label{dim3}
Assume that $\dim X = 3$.
Then any $K$-equivalent birational map between projective $\mathbf{Q}$-factorial  KLT pairs
$f: (X,B) \dashrightarrow (Y,C)$ which is an isomorphism in codimension $1$
is decomposed into a sequence of flops.
\end{Prop}

\begin{proof}
Assume that $(L_X,l) \le 0$ for a curve $l$.
Since $\dim X = 3$, there are only finitely many such curves $l$.
Then $l$ is the image of a fiber of $q$ by $p$.
Therefore we have $((K_X+B),l)=0$.
\end{proof}

\section{Derived McKay correspondence}

We prove Theorems~\ref{abelian} and \ref{GL2} as well as Theorem~\ref{relative}~(a) 
in this section.

\begin{proof}[proof of Theorem~\ref{abelian}]
Let $B$ be a $\mathbf{Q}$-divisor on $X$ such that $p^*(K_X+B)=K_{\mathbf{A}}$.
Since the action of $G$ is diagonalizable, the pair $(X,B)$ is toric.
We have $D^b(\text{coh}([\mathbf{A}/G]) \cong D^b(\text{coh}(\tilde X))$ for the 
smooth Deligne-Mumford stack $\tilde X$ associated to the pair $(X,B)$.

There exists a toric relative minimal model of $X$.
Any other relative minimal model is obtained by a sequence of flops from the toric 
relative minimal model by \cite{flop}. 
Since the extremal rays are toric, the sequence is automatically toric, 
hence any relative minimal model is again toric.
By Theorem~\ref{relative}, it is sufficient to prove the theorem for a particular relative minimal model.

Since $X$ is $\mathbf{Q}$-factorial and $(X,B)$ is KLT, we can construct a sequence 
toric birational morphisms
\[
(X_t,B_t) \to (X_{t_1},B_{t-1}) \to \dots \to (X_0,B_0)=(X,B)
\]
which satisfy the following conditions:

(1) Each step $f_i: (X_i,B_i) \to (X_{i-1},B_{i-1})$ is a divisorial contraction such that 
$B_i = f_{i*}^{-1}B_{i-1}$ and $K_{X_i}+B_i \le K_{X_{i-1}}+B_{i-1}$ for $1 \le i \le t$.

(2) The pair $(X_t,B_t)$ is terminal.

Next we run a toric minimal model program for $X_t$ over $X$ 
to obtain a sequence of birational maps
\[
X_t=Y_0 \dashrightarrow Y_1 \dashrightarrow \dots \dashrightarrow Y_s=Y
\]  
consisting of divisorial contractions and flips such that $K_Y$ is nef over $X$. 
We have $K_{Y_{j-1}} > K_{Y_j}$ for $1 \le j \le s$.

We apply Theorem~\ref{relative} and Corollary~\ref{combine} 
to each step of the above sequences to obtain our result. 
\end{proof}

\begin{proof}[proof of  Theorem~\ref{GL2}]
Let $H$ be the normal subgroup of $G$ defined by $H = G \cap SL(2,\mathbf{C})$.
The residue group is a cyclic group $\mathbf{Z}_r=G/H$ for $r=[G:H]$.
By \cite{BKR}, 
a component of the Hilbert scheme of $H$-invariant subschemes
of $\mathbf{A}$ gives a crepant resolution $f: M \to \mathbf{A}/H$, and 
there is a derived equivalence $\Phi:D^b(\text{coh}(M)) \to D^b(\text{coh}^H(\mathbf{A}))
= D^b(\text{coh}([\mathbf{A}/H]))$.
More precisely, $M$ is the moduli space of closed subschemes
$P \subset \mathbf{A}$ which is $H$-invariant
and such that $\text{length }\mathcal{O}_P = \#H$.
For a generic point $[P] \in M$, $P$ is a reduced subscheme, 
and corresponds to a generic orbit of $H$ in $\mathbf{A}$.
There exists a universal subscheme 
$W \subset M\times \mathbf{A}$ with projections 
$p: W \to M$ and $q: W \to \mathbf{A}$ such that the 
functor $\Phi=q_*p^*$ gives the equivalence.

For $[P] \in M$ and $g \in G$, $gP$ is also $H$-invariant, since $H$ is a normal subgroup.
Hence $[gP] \in M$, and $G$ acts on $M$ such that its subgroup $H$ acts trivially.
The equivariant derived category $D^b(\text{coh}^{\mathbf{Z}_r}(M))$ is identified with 
the category where the objects are $G$-equivariant complexes on which $H$ acts 
trivially, and the morphisms are $G$-equivariant morphisms.

$G$ acts on the product $M \times \mathbf{A}$ diagonally and preserves the subscheme $W$.
We claim that the functor
\[
\Phi'= q_*p^*: D^b(\text{coh}^{\mathbf{Z}_r}(M)) \to D^b(\text{coh}^G(\mathbf{A})).
\]
is an equivalence.
Indeed if $a,b$ are $\mathbf{Z}_r$-equivariant objects, then
the isomorphism 
\[
\text{Hom}_M(a,b) \cong \text{Hom}_{\mathbf{A}}(\Phi(a),\Phi(b))^H
\]
implies an isomorphism
\[
\text{Hom}_M(a,b)^{\mathbf{Z}_r} \cong 
(\text{Hom}_{\mathbf{A}}(\Phi(a),\Phi(b))^H)^{\mathbf{Z}_r}
= \text{Hom}_{\mathbf{A}}(\Phi(a),\Phi(b))^G. 
\]

Let $g: M \to M/\mathbf{Z}_r$ be the map to the quotient space.
We define a boundary divisor $B_{M/\mathbf{Z}_r}$ on $M/\mathbf{Z}_r$ by 
$g^*(K_{M/\mathbf{Z}_r}+B_{M/\mathbf{Z}_r})=K_M$. 
Since $M$ is smooth, $M/\mathbf{Z}_r$ has only cyclic quotient singularities.
Let $h: Y' \to M/\mathbf{Z}_r$ be the minimal resolution, and 
let $k: Y' \to Y$ be the contraction morphism to the minimal resolution of $\mathbf{A}/G$.
The former is a toroidal morphism, and decomposed into toroidal divisorial contractions.
The latter is a composition of divisorial 
contractions of $(-1)$-curves, which are toroidal too.

The canonical divisors satisfy the following inequalities:
\[
K_Y \le K_{Y'} \le K_{M/\mathbf{Z}_r} \le K_{M/\mathbf{Z}_r}+B_{M/\mathbf{Z}_r}
=K_{\mathbf{A}/G}.
\]
By Theorem~\ref{relative}, we obtain fully faithful embeddings
\[
D^b(\text{coh}(\tilde Y)) \to D^b(\text{coh}(\tilde Y')) \to D^b(\text{coh}([M/\mathbf{Z}_r])) 
= D^b(\text{coh}([\mathbf{A}/G]))
\]
whose semi-orthogonal complements are generated by the derived categories of subvarieties of 
$X=\mathbf{A}/G$.
\end{proof}

\begin{proof}[proof of Theorem~\ref{relative}~(a)]
It is sufficient to consider the following situation:
$X=Y$ is an affine toric variety corresponding to a simplicial cone 
$\sigma \subset N_{\mathbf{R}}$ generated by primitive vectors $v_1,\dots,v_n \in N$.
Let $B_i$ be prime divisors on $X$ corresponding to the $v_i$, and let 
$B = \sum (1-1/r_i)B_i$ and $C=\sum (1-1/s_i)B_i$, where $r_1 > s_1$ and
$r_i=s_i$ for $i=2,\dots,n$.
Let $t_1=\text{LCM}(r_1,s_1)$ and $t_i=r_i$ for $i=2,\dots,n$.

The stacks $\tilde X$,  $\tilde Y$ and $\tilde W$ are defined by the sublattices 
$N_X$, $N_Y$ and $N_W$ 
of $N$ generated by the $r_iv_i$, $s_iv_i$ and $t_iv_i$, respectively. 
Let $\tilde B_i^X$, $\tilde B_i^Y$ and $\tilde B_i^W$ be prime divisors 
on $\tilde X$,  $\tilde Y$ and $\tilde W$ over $B_i$, respectively.
They correspond primitive vectors $r_iv_i$, $s_iv_i$ and $t_iv_i$, respectively.
Let $p: \tilde W \to \tilde X$ and $q: \tilde W \to \tilde Y$ be the natural morphisms.
Then the fully faithful functor $\Phi: D^b(\text{coh}(\tilde Y)) \to D^b(\text{coh}(\tilde X))$ is 
constructed as $\Phi=p_*q^*$ in \cite{log crepant}~Theorem~4.2 (1).
We have 
\[
\Phi(\mathcal{O}_{\tilde Y}(k\tilde B_1^Y)) = 
\mathcal{O}_{\tilde X}(\llcorner kr_1/s_1 \lrcorner \tilde B_1^X)
\]
for $k = 1,\dots, s_1-1$.

We consider the semi-orthogonal complement.
It is generated by the sheaves
$\mathcal{O}_{\tilde B_1^X}(l\tilde B_1^X)$ for $1 \le l \le r_1-1$ such that
$l \ne \llcorner kr_1/s_1 \lrcorner$ for any $k$.
We have 
\[
R\text{Hom}_{\tilde X}(\mathcal{O}_{\tilde X}(\llcorner kr_1/s_1 \lrcorner \tilde B_1^X),
\mathcal{O}_{\tilde B_1^X}(l\tilde B_1^X))=0
\]
for such $k,l$, and 
\[
R\text{Hom}_{\tilde X}(\mathcal{O}_{\tilde B_1^X}(l\tilde B_1^X), 
\mathcal{O}_{\tilde B_1^X}(l'\tilde B_1^X))=0
\]
for such $l,l'$ with $l < l'$.

We define a toric boundary $D$ on $B_1$ 
from the pair $(X,B)$ as follows (cf. \cite{toric}~\S5).
Let $\bar N = N/\mathbf{Z}v_1$, and write $v_i + \mathbf{Z}v_1 = d_i \bar v_i \in \bar N$ for 
primitive vectors $\bar v_i$ ($i=2,\dots,n$).
Then define 
$D= \sum_{i=2}^n (1-1/d_ir_i)D_i$ for $D_i=B_1 \cap B_i$. 
Let $\tilde B_1$ be the associated smooth Deligne-Mumford stack above $(B_1,D)$.
Let $\tilde B_1'$ be the normalization of the fiber product $\tilde B_1 \times_X \tilde X$ with 
natural morphisms 
$p_1: \tilde B_1' \to \tilde B_1$ and $p_2: \tilde B_1' \to \tilde X$.
Then as in \cite{toric}, the semi-orthogonal complement is generated by 
subcategories 
\[
p_{2*}p_1^*D^b(\text{coh}(\tilde B_1)) \otimes \mathcal{O}_{\tilde X}(l\tilde B_1^X)
\]
$1 \le l \le r_1-1$ such that $l \ne \llcorner kr_1/s_1 \lrcorner$ for any $k$.
Moreover, there is no morphism between objects belonging to different $k$'s.
Therefore we conclude the proof.
\end{proof}

Graduate School of Mathematical Sciences, University of Tokyo,

Komaba, Meguro, Tokyo, 153-8914, Japan

kawamata@ms.u-tokyo.ac.jp

\end{document}